\documentclass[12pt, twoside, leqno]{article}

\usepackage{amsmath,amsthm}
\usepackage{amssymb}

\usepackage{enumerate}

\usepackage{graphicx}

\newtheorem{thm}{Theorem}[section]
\newtheorem{cor}[thm]{Corollary}

\theoremstyle{definition}

\numberwithin{equation}{section}

\begin{document}

%%%%% To ease editing, for IMPAN journals add:

\baselineskip=17pt

%%%%%%%%%%%%%%%%

\title{On the Axiomatic Systems of Steenrod Homology Theory of Compact Spaces}

\author{Anzor Beridze\\
Department of Mathematics\\
Batumi Shota Rustaveli State University\\
35, Ninoshvili St., Batumi 6010, Georgia\\ 
E-mail: a.beridze@bsu.edu.ge
\and 
Leonard Mdzinarishvili\\
Department of Mathematics\\
Georgian Technical University\\
77, Kostava St., Tbilisi 0171, Geogia\\
E-mail: l.mdzinarishvili@gtu.ge}

\date{}

\maketitle

%% Classification and key words; note that the 2010 classification is used:

\renewcommand{\thefootnote}{}

\footnote{2010 \emph{Mathematics Subject Classification}: 	55N07; 	55N40.}

\footnote{\emph{Key words and phrases}: Partial Continuity, Nontrivial Extension, the Uniqueness Theorem, the Universal Coefficient Formula, Exact Bifunctor, Injective Group.}

\renewcommand{\thefootnote}{\arabic{footnote}}
\setcounter{footnote}{0}

%%%%%%%%

\begin{abstract}
On the category of compact metric spaces an exact homology theory was defined  and its relation to the Vietoris homology theory was studied by Steenrod \cite{10}. In particular, the homomorphism from the Steenrod homology groups to the Vietoris homology groups was defined and it was shown that the kernel of the given homomorphism are homological groups, which was called weak homology groups \cite{10}, \cite{11}. The Steenrod homology theory on the category of compact metric pairs was axiomatically described by J.Milnor. In  \cite{9} the uniqueness theorem is proved using the Eilenberg-Steenrod axioms and as well as relative homeomorphism and clusres axioms. J. Milnor constructed the homology theory on the category $Top^2_C$ of compact Hausdorff pairs and proved that on the given category it satisfies nine axioms  - the Eilenberg-Steenrod, relative homeomorphis and cluster axioms (see theorem 5 in \cite{9}). Besides, using the construction of weak homology theory, J.Milnor proved that constructed homology theory satisfies partial continuity property on the subcategory $Top^2_{CM}$ (see theorem 4 in \cite{9}) and the universal coefficient formula on the category $Top^2_C$ (see Lemma 5 in \cite{9}).   
On the category of compact Hausdorff pairs, different axiomatic systems were proposed by N. Berikashvili \cite{1}, \cite{2}, H.Inasaridze  and L. Mdzinarishvili \cite{6}, L. Mdzinarishvili \cite{8} and H.Inasaridze  \cite{5}, but there was not studied any connection between them. The paper studies this very problem. In particular, in the paper it is proved that any homology theory in Inasaridze sense is the homology theory in the Berikashvili sense, which itself is the homology theory in the Mdzinarishvili sense. On the other hand, it is shown that if a homology theory in the Mdzinarishvili sense is exact functor of the second argument, then it is the homology in the Inasaridze sense.
\end{abstract}

\section{Introduction}

Let $Top^2_C$ be the category of compact Hausdorff pairs and continuous maps and $ {\mathcal Ab}$ be the category of abelian groups.

A sequence $\bar{H}_*=\{ \bar{H}_n \}_{n \in Z}$ of covariant functors $ \bar{H}_n:Top^2_C \to Ab$ is called homological \cite{8}, \cite{3}, if:

$1_H$) for each object $(X,A) \in Top^2_C$ and $n \in Z$ there exists a $\partial$-homomorphism
\begin{equation}
\partial: \bar{H}_n(X,A) \to \bar{H}_{n-1} (A)
\end{equation}
($\bar{H}_{n} (A) \equiv \bar{H}_{n-1} (A, \emptyset)$, wehere $\emptyset$ is the empty set);

$2_H$) the diagram  
\begin{equation}
\begin{matrix}
    \bar{H}_n(X,A) ~~\to ~~\bar{H}_{n-1} (A;G)  \\
  ~~~~ \downarrow f_* ~~~~~~~~\downarrow (f_{|A})_*   \\
   \bar{H}_n(Y,B) ~~\to ~~\bar{H}_{n-1} (B;G)   \\
\end{matrix}.
\end{equation}
is commutative for each continuous mapping $f: (X,A) \to (Y,B)$ ($f_* : \bar{H}_n(X,A) \to \bar{H}_n(Y,B)$ and $(f_{|A})_*  : \bar{H}_n(A) \to \bar{H}_n(B)$ are  the homomorphisms induced by $f: (X,A) \to (Y,B)$ and $f_{|A}:A \to B$, correspondingly).

Let ${\mathcal Pol^2}$ be the full subcategory of the category $Top^2_C$, consisting of compact polyhedral pairs. 

A homological sequence $\bar{H}_*=\{ \bar{H}_n\}_{n \in Z}$ defined on the category  $Top^2_C$ is called homology theory in the Eilenberg-Steenrod sense if it satisfies homotopy, excision, exactness and dimension axioms \cite{3}. It is known that up to an isomorphism such a homology theory is unique on the subcategory ${\mathcal Pol^2}$ of compact polyhedral pairs  \cite{3} and it is denoted by ${H}_*=\{{H}_n\}_{n \in Z}$, but it is not unique on the category $Top^2_C$ of compact Hausdorff pairs.
   
The Steendor homology theory on the category of compact metric  pairs was first axiomatically described by J. Milnor \cite{9}. He proved the uniqueness theorem using the Eilenberg-Steenrod axioms and additionally two more - relative homeomorphism and clustes axioms:

RH (relative homeomorphism axiom): if $ f :(X,A) \to (Y,B)$ is a map in $Top_{CM}^2$ which carries $X-A$ homeomorphically onto $Y-B$, then

\begin{equation}
f_*:\bar{H}_n(X,A) \to \bar{H}_n(Y,B)
\end{equation} 
is an isomorphism.

CL(cluster axiom): if $X$ is the union of countable many compact subsets $X_1, X_2, \dots $which intersect pairwise at a single point $*$, and which  have diameters tending to zero, then $\bar{H}_n(X,*)$ is naturally isomorphic to the direct product of the groups $\bar{H}_n(X_i,*)$.

In \cite{9} the following is proved:

\begin{thm} (see theorem 3 in \cite{9}) Given two homology theories $\bar{H}_*^M$ and $\bar{H}_*$ on the category $Top_{CM}^2$, both satisfying the nine axioms (the Eilenberg-Steenrod, relative homeomorpism and cluster axioms), any coefficient isomorpism $\bar{H}_0^M(*) \approx \bar{H}_0(*) \approx G$ extends uniquely to an equivalence between the two homology theories.
\end{thm}

In \cite{9} J. Milnor constructed the homology theory $\bar{H}_*^M$ on the category $Top_C^2$ of compact Hausdorff pairs and gave its several properties. In particular \cite{9} the following is proved:

\begin{thm} (see theorem 5 in \cite{9}) The homology theory $\bar{H}_*^M$, defined on the category $Top^2_C$ of compact Hausdorff pairs, satisfies the nine axioms (the Eilenberg-Steenrod axioms as well as relative homeomorphism and clusre axioms).

\end{thm}

\begin{thm} (see theorem 4 in \cite{9}) Let $\bar{H}_*^M$ be a homology theory satisfying the nine axioms (the Eilenberg-Steenrod axioms as well as relative homeomorphism and cluster axioms), and let $ X_1 \leftarrow X_2 \leftarrow X_3 \leftarrow \dots $ be an inverse system of compact metric spaces with inverse limit $X$. Then there is an exact sequence
\begin{equation}
0 \to {\varprojlim}^1 \bar{H}_{n+1}^M(X_i) \to \bar{H}_n^M(X) \to \varprojlim \bar{H}_n^M(X_i) \to 0
\end{equation}
for each integer $n$. A corresponding assertion holds if each space is replaced by a pair. 
\end{thm}

\begin{thm} (see lemma 5 in \cite{9}) The homology theory $\bar{H}_*^M$ is related to the $\check{C}$ech cohomology theory by a split exact sequence

\begin{equation}
0 \to Ext( \check{H}^{n+1}(X,A);G ) \to \bar{H}_n^M(X,A;G) \to Hom(\check{H}^n(X,A);G) \to 0.
\end{equation}

\end{thm}

As we see the uniqueness theorem was proved only on the  category $Top_{CM}^2$ of compact metric pairs \cite{9} and therefore, the problem was open for the category $Top_C^2$. 

\ \

The axiomatic description of the Steenrod homology theory on the category $Top_C^2$ of compact Hausdorff pairs was given by Berikashvili \cite{1}, \cite{2}. In particular, in \cite{1} it is proved that if a homological sequence $\bar{H_*}=\{ \bar{H_n} \}_{n \in Z} $ defined on the category $Top_C^2$ of compact Hausdorff pairs satisfies the Eilenberg-Steenrod axioms and the following A, B and C axioms, then it is naturally isomorphic to the Chogoshvili homology theory:

\ \

A: The projection $(X,A) \to (X/A,*)$ induces an isomorphism $\bar{H}_n(X,A) \approx \bar{H}_*(X/A,*)$.

\ \
B: For the inverse spectrum of pairs $\{ (S_\alpha^n,*), \pi_{\alpha \beta} \}$, where  $S_\alpha^n$ is a finite cluster of $n$-dimension spheres and $\pi_{\alpha \beta}$ maps  each sphere of the cluster either to the fixed piont or homeomorphically to a cluster sphere, there holds the equality 
\begin{equation}
\bar{H}_*( \varprojlim \{ (S_\alpha^n,*), \pi_\beta^\alpha \}) \approx \varprojlim \{ \bar{H}_*(S_\alpha^n,*), (\pi_{\alpha\beta})_* \},   ~~~~~  n \in Z.
\end{equation}

\ \

C: The natural homomorphism 
\begin{equation}
\varinjlim \bar{H}_n(|\mathcal{N}(X)|_p) \to \bar{H}_n(X), ~~~~~ n \in Z,
\end{equation}
induced by the mapping $\omega:|\mathcal{N}(X)| \to X$, were  $|\mathcal{N}(X)|$ is the limit of the inverse spectrum of realizations of complexes of the spectrum $\mathcal{N}(X)=\{ \mathcal{N}_\lambda (X), \pi_{\lambda \mu} \}$ ($\mathcal{N}_\lambda (X)$ is the nerve of a finite closed covering obtained from a finite partitioning of $X$ \cite{1}) and $|\mathcal{N}_\lambda (X)|_p=\varprojlim \{ |\mathcal{N}_\lambda (X)|^p, \pi_{\lambda \mu} \}\subset |\mathcal{N}(X)| $ ($K^p$ denotes the $p$-skeleton of the complex $K$), is an isomorphisms.

\ \

In \cite{2} Berikashvili proposed new $C_1$ and $C_2$ axioms and  the universal coefficient formula as one more new axiom as well:

\ \

$C_1$: If a continuous map $f: X \to Y$ induces an isomorphism $f^* : \check{H}^n(Y;Z) \to \check{H}^n(X;Z)$ for $n < p$, then for $n < p-1$ a homomorphism $f_* : \bar{H}_n(X;Z) \to \bar{H}_n(Y;Z)$ is an isomorphism as well. 

 \ \

$C_2$: If a continuous map $f: X \to Y$ is surjective and  $\bar{H}_n(f^{-1}(y),*)=0$ for each $y \in Y$ and $n<p$, then for $n < p$ a homomorphism $f_* : \bar{H}_n(X;Z) \to \bar{H}_n(Y;Z)$ is an isomorphism.

\ \

$D$: For each pairs $(X,A)$ there exists  a functorial exact sequence

\begin{equation}
0 \to Ext(\check{H}^{n+1}(X,A);G) \to \bar{H}_n(X,A;G) \to Hom(\check{H}^n(X,A);G) \to 0,
\end{equation}
where $G=\bar{H}_0(*)$.

Consequently, in \cite{2} the following is proved:

\begin{thm} (see theorem 4.4 \cite{2}) The Steenrod-Sitnikov homology theory defined on the category $Top_{C}^2$ of compact Hausdorff pairs with coefficients any module $G$ uniquely is characterized by the Eilenberg-Steenrod axioms with one of the following 4 systems of axioms: 1) A, B, C axioms; 2)axiom D; 3) A, B, $C_1$ axioms; 4) A, B, $C_2$  axioms for the finite generated abelian group.
\end{thm}

In \cite{6} H. Inasaridze and L. Mdzinarishvili gave one more different axiomatic system using the modified form of the continuity axiom, as it is called, partial continuity axiom:

PC (partial continuity axiom):  Let $(X,A)$ be the inverse limit of inverse system $\{ (X_\lambda ,A_\lambda ), p_{\lambda,\lambda'}, \Lambda \}$ of compact polyhedra, then for each integer $n$ there is a functorial exact sequence
\begin{equation}
0 \to {\varprojlim}^1 H_{n+1}(X_\lambda ,A_\lambda) \to \bar{H}_n(X,A) \to {\varprojlim} H_{n}(X_\lambda ,A_\lambda) \to 0.
\end{equation}

Using the partial continuity axiom in the paper  \cite{8} L. Mdzinarishvili defined a nontrivial external extension  $\bar{H}_*$ of the homology theory  $H_*$ defined on the category ${\mathcal Pol^2}$ of compact polyhedra pairs to the category $Top_C^2$ of compact Hausdorff pairs \cite{8}. In particular, homological sequence  $\bar{H}_*$ defined on the category $Top_C^2$ is called extension of homology theory  $H_*$ (which is unique up to an isomorphism) defined on the category ${\mathcal Pol^2}$, if  on the subcategory ${\mathcal Pol^2}$ it is equivalent to $H_*$ \cite{8}. The homological sequence  $\bar{H}_*= \{ \bar{H}_n \}_{n \in Z} $ defined on the category $Top_C^2$ is called a nontrivial external extension of the homology theory ${H}_*= \{{H}_n \}_{n \in Z}$ defined on the category  ${\mathcal Pol^2}$, if the following conditions are fulfilled:

$1_{NT}$)  $\bar{H}_*$ is an extension of the homology theory  $H_*$;

$2_{NT}$) the exact sequence 
\begin{equation}
0 \to {\varprojlim}^1 H_{n+1}(X_\lambda ,A_\lambda) \to \bar{H}_n(X,A) \to {\varprojlim} H_{n}(X_\lambda ,A_\lambda) \to 0
\end{equation}
holds for any object $(X,A) \in Top^2_C$, any inverse system $\{ (X_\lambda ,A_\lambda ), p_{\lambda,\lambda'}, \Lambda \}$ of compact polyhedra such that $(X,A)=\varprojlim \{ (X_\lambda ,A_\lambda ), p_{\lambda,\lambda'}, \Lambda \}$ and $n \in Z$;

$3_{NT}$) The commutative diagram  
\begin{equation}
\begin{matrix}
    {\varprojlim}^1 \bar{H}_{n+1}(X_\lambda ,A_\lambda) ~~\to ~~\bar{H}_n (X,A)  \\
  \downarrow {\varprojlim}^1 \tilde{f}_* ~~~~~~~~~~~~~~~~~~~~\downarrow f_* ~~~~~~  \\
    {\varprojlim}^1 \bar{H}_{n+1}(Y_\gamma ,B_\gamma) ~~\to ~~\bar{H}_n (Y,B)   \\
\end{matrix},
\end{equation}
holds for any continuous mapping $f:(X,A) \to (Y,B)$ from $Top^2_C$, where $\tilde{f}_* : \{ H_{n}(X_\lambda ,A_\lambda),(p_{\lambda,\lambda'})_*, \Lambda \} \to \{ H_{n}(Y_\gamma ,B_\gamma), (q_{\gamma,\gamma'})_*, \Gamma \}$ is mapping of the inverse systems;

$4_{NT}$) $\bar{H}_*$ satisfies the exactness axiom.

In \cite{8} a homological sequence $\bar{H}_*$ defined on the category $Top_C^2$ of compact Hausdorff pairs is called:

1) a homology theory in the Eilenberg-Steenrod sense if it satisfies the axioms of homotopy, excision, exactness and dimension;

2) a homology theory in the Milnor sense if it satisfies the axioms of homotopy, excision, exactness, dimension, relative homeomorphism and cluster axioms;

3) a homology theory in the Berikashvili sense if it satisfies the axioms of homotopy, excision, exactness, dimension and A, B and C axioms.

In \cite{8} it is shown that any nontrivial external extension is homology theory in the Eilenberg-Steenrod, in the Milnor as well as in the Berikashvili sense:

\begin{thm} (see theorem 1.2 in \cite{8}) if $\bar{H}_*$ is a nontrivial external extension of the homology theory  $H_*$  to the category $Top_{C}^2$, then it is a theory in the Eilenberg-Steenrod sense.
\end{thm}

\begin{thm} (see theorem 1.3 in \cite{8}) if $\bar{H}_*$ is a nontrivial external extension of the homology theory  $H_*$, defined on the category $Top_{C}^2$, then it is a homology theory in the Milnor sense.
\end{thm}

By theorem 1.2 and theorem 1.7 the following is true:

\begin{cor} On the category $Top^2_{CM}$ of compact metric pairs the $\bar{H}_*$ is a nontrivial external extension if and only if it is the homology theory in the Milnor sense.
\end{cor}

\begin{thm} (see theorem 1.5 in \cite{8}) if $\bar{H}_*$ is a nontrivial external extension of the homology theory  $H_*$ defined on the category $Top_{C}^2$, then it is a homology theory in the Berikashvili sense.
\end{thm}

Consequently, the following is obtained:

\begin{cor} (see corollary 1.4 in \cite{8}) Any nontrivial external extension of the homology theory $H_*$ defined on the category $Top_{C}^2$ is isomorphic to the Steenrod homology theory. 
\end{cor}

Note that there are many different constructions of an exact homology theory, but all of them are functors of the second argument as well: For each short exact sequence
\begin{equation}
0 \to G \to G' \to G'' \to 0,
\end{equation} 
there is the functorial natural long exact sequence:
\begin{equation}
\dots \to \bar{H}_{n+1}(X;G'') \to \bar{H}_n(X;G) \to \bar{H}_n(X;G') \to \bar{H}_n(X;G') \to \dots.
\end{equation}
Therefore, for a homology theory it is natural to consider it as a bifunctor. In the paper \cite{5} H. Inasaridze described  exact bifunctor homology theory using the continuity property for the infinitely divisible abelian groups. In particular, \cite{5} the following is proved:

\begin{thm} (see theorem 1 in \cite{5}) There exists one and only one exact bifunctor homology theory on the category  $Top_C^2$ of compact Hausdorff pairs with coefficients in the category of abelian groups (up to natural equivalence) which satisfies the axioms of homotopy, excision, dimension, and continuity for every infinitely divisible group.
\end{thm}

Therefore, for the Steendor homology theory there are different axiomatic systems, but it is not known what the relation between them is and which one is the minimal one in the axiomatic sense. In the second part we will study this problem.

\section{Relations between Different Axiomatic Systems}

In this paper we will say that a homological sequence $\bar{H}_*$ defined on the category $Top_C^2$ of compact Hausdorff pairs is:

1) a homology theory in the Berikashvili sense if it satisfies the axioms of homotopy, excision, exactness, dimension and axiom D (The Universal Coefficient Formula);

2)  a homology theory in the Mdzinarishvili sense if it is a nontrivial external extension;

3)  a homology theory in the Inasaridze sense if it is the exact functor of the second argument  and satisfies the axioms of homotopy, excision, exactness, dimension and continuity for every infinitely divisible group.

Note that in the category $\mathcal{A}b$ of abelian groups G is infinitely divisible group if and only if it is injective. Therefore, in our paper instead of the term "infinitely divisible group" we will use "injective group".

\begin{thm} If $\bar{H}_*$ is a homology  theory in the Inasaridze sense, defined on the category $Top_C^2$ of compact Hausdorff pairs,  then it is the homology theory in the Berikashvili sense.
\end{thm} 
\begin{proof} For each compact Hausdorff space $X \in Top_C^2$ consider an inverse system $ {\bf {X}} = \{ X_{\lambda}, p_{\lambda, \lambda'}, \Lambda \}$ of compact polyhedra such that 
$X={\varprojlim}{\bf {X}}$. By the condition of the theorem, for each injective group $G$ we have an isomorphism:
\begin{equation}
 \bar{H}_n( X;G)=\bar{H}_n(\varprojlim X_\lambda;G) \approx\varprojlim H_n(X_\lambda;G).
\end{equation}
For each compact polyhedra $X_\lambda$ we have the exact sequence \cite{7}:
\begin{equation}
0 \to Ext(H^{n+1}(X_{\lambda});G) \to H_n (X_{\lambda};G) \to Hom(H^n(X_{\lambda});G) \to 0,
\end{equation}
which induces the long exact sequence:

 $0 \to \varprojlim Ext(H^{n+1}(X_{\lambda});G) \to \varprojlim H_n (X_{\lambda};G) \to \varprojlim Hom(H^n(X_{\lambda});G) \to$
 
 $ {\varprojlim}^1 Ext(H^{n+1}(X_{\lambda});G) \to {\varprojlim}^1 H_n (X_{\lambda};G) \to {\varprojlim}^1 Hom(H^n(X_{\lambda});G) \to$ 
 \begin{equation}
{\varprojlim}^2 Ext(H^{n+1}(X_{\lambda});G) \to {\varprojlim}^2 H_n (X_{\lambda};G) \to {\varprojlim}^2 Hom(H^n(X_{\lambda});G) \to  \dots.
 \end{equation}
  
Note that for each injective group $G$ the functor $Ext(-;G)$ is trivial and by (2.3) we obtain the isomorphism:
\begin{equation}
 ~~ {\varprojlim}H_n(X_\lambda;G) \approx {\varprojlim} Hom(H^n(X_\lambda);G).
\end{equation}
If we apply the isomorphism ${\varprojlim} Hom(H^n(X_\lambda);G) \approx Hom ({ \varinjlim} H^n(X_\lambda);G)$, then by (2.4) we obtain:
\begin{equation}
 ~~ {\varprojlim}H_n(X_\lambda;G) \approx Hom({\varinjlim} H^n(X_\lambda);G)=Hom( \check{H}^n(X);G).
\end{equation}
Therefore, by (2.1) if $G$ is an injective then
\begin{equation}
\bar{H}_n(X;G) \approx \varprojlim H_n(X_\lambda;G) \approx Hom(  \check{H}^n(X);G).
\end{equation}
Now consider any abelian group $G$ and corresponding injective resolution:
\begin{equation}
 ~~0 \to G \to G' \to G'' \to 0.
\end{equation}

Apply to the sequence (2.7) by the functor $Hom(\check{H}^n(X);-)$. The groups $G'$ and $G''$ are injective and so we have: 
$$0 \to Hom(\check{H}^n(X);G)  \to  Hom(\check{H}^n(X);G') \to$$
\begin{equation}
   Hom(\check{H}^n(X);G'') \to  Ext(\check{H}^n(X);G) \to 0.
\end{equation}
Therefore, for each integer $n \in N$ we have
\begin{equation}
 Ker(Hom(\check{H}^n(X):G') \to  Hom(\check{H}^n(X):G'')) \approx Hom(\check{H}^n(X);G),
\end{equation}
\begin{equation}
 Coker(Hom(\check{H}^n(X):G') \to  Hom(\check{H}^n(X):G'')) \approx Ext(\check{H}^n(X);G).
\end{equation}
Now apply sequence (2.7) by homological bifunctor $\bar{H}_*(X;-)$, which gives the following long exact sequence:
\begin{equation} \dots  \to \bar{H}_{n+1}(X;G') \to  \bar{H}_{n+1}(X;G'') \to  \bar{H}_{n}(X;G) \to \bar{H}_{n}(X;G') \to  \bar{H}_{n}(X;G'') \to \dots.
\end{equation}
Therefore, for each $n \in N$ we obtain the following short exact sequence:
$$0  \to Coker(\bar{H}_{n+1}(X;G') \to  \bar{H}_{n+1}(X;G'')) \to$$
\begin{equation}   \bar{H}_{n}(X;G) \to Ker(H_{n}(X;G') \to  \bar{H}_{n}(X;G'')) \to 0.
\end{equation}
By (2.6), (2.9), (2.10) and (2.12)  we finally obtain the following short exact sequence:
\begin{equation} 0  \to Ext( \check{H}^{n+1}(X;G)) \to  \bar{H}_{n}(X;G) \to Hom(\check{H}^{n}(X;G) \to  0.
\end{equation}
\end{proof}

\begin{thm} If $\bar{H}_*$ is a homology  theory in the Berikashvili sense, defined on the category $Top_C^2$ of compact Hausdorff pairs,  then it is a homology theory in the Mdzinarishvili sense.
\end{thm} 

\begin{proof} For each compact Hsusdorff space $X \in Top_C^2$ consider an inverse system $ {\bf {X}} = \{ X_{\lambda}, p_{\lambda, \lambda'}, \Lambda \}$ of compact polyhedra such that 
$X={\varprojlim}{\bf {X}}$. For each $\lambda \in \Lambda$ and any abalian group $G$ consider the following exact sequence \cite{7}:
\begin{equation}
0 \to Ext({H}^{n+1}(X_{\lambda});G) \to H_n (X_{\lambda};G) \to Hom({H}^n(X_{\lambda});G) \to 0,
\end{equation}
which induces the long exact sequence 

 $0 \to \varprojlim Ext({H}^{n+1}(X_{\lambda});G) \to \varprojlim H_n (X_{\lambda};G) \to \varprojlim Hom({H}^n(X_{\lambda});G) \to$
 
 $ {\varprojlim}^1 Ext({H}^{n+1}(X_{\lambda});G) \to {\varprojlim}^1 H_n (X_{\lambda};G) \to {\varprojlim}^1 Hom({H}^n(X_{\lambda});G) \to$ 
 \begin{equation}
{\varprojlim}^2 Ext({H}^{n+1}(X_{\lambda});G) \to {\varprojlim}^2 H_n (X_{\lambda};G) \to {\varprojlim}^2 Hom({H}^n(X_{\lambda});G) \to  \dots.
 \end{equation}
Note that for each $\lambda \in \Lambda$ the cohomology group $H^n(X_{\lambda};G)$ is finitely generated \cite{7} and so by Corollary 1.5 in \cite{4} we have:
\begin{equation}
{\varprojlim}^r Ext({H}^{n}(X_{\lambda});G)=0, ~~r \geq 1. 
\end{equation}
Therefore, by (2.15) and (2.16) we obtain the exact sequence:
\begin{equation} 
0 \to \varprojlim Ext({H}^{n+1}(X_{\lambda});G) \to \varprojlim H_n (X_{\lambda};G) \to \varprojlim Hom({H}^n(X_{\lambda});G) \to 0.
\end{equation}
Naturally, there exists the commutative diagram:
\begin{equation}
\begin{matrix}
   0 \to Ext(\check{H}^{n+1}(X);G) ~~\to ~~\bar{H}_n (X_;G)~~ \to ~~ Hom(\check{H}^n(X);G) ~~\to ~~0  \\
   \downarrow \psi ~~~~~~~~~~~~~~~~~~~~~~~~~\downarrow \varphi  ~~~~~~~~~~~~~~~~~~~~\downarrow \chi  \\
   0 \to \varprojlim Ext({H}^{n+1}(X_{\lambda});G) \to \varprojlim H_n (X_{\lambda};G) \to \varprojlim Hom({H}^n(X_{\lambda});G) \to 0  \\
\end{matrix}.
\end{equation}
On the other hand, $ \chi : Hom(\check{H}^n(X);G)=Hom( \varinjlim H^n(X_\lambda);G) \to \varprojlim Hom({H}^n(X_{\lambda});G)$ is an isomorphism and so
\begin{equation}
Ker \psi \approx Ker \varphi,    ~~~ Coker \psi \approx Coker \varphi.
\end{equation}
Therefore, we have the following commutative diagram of the exact sequences:
\begin{equation}
\begin{matrix}
  0 \to  Ker \psi \to Ext(\check{H}^{n+1}(X);G) ~~\to ~~\varprojlim Ext({H}^{n+1}(X_\lambda);G)~~ \to ~~ Coker \psi ~~\to ~~0   \\
   \downarrow  \approx ~~~~~~~~~~~~~\downarrow   ~~~~~~~~~~~~~~~~~~~~~~~~~~~~~~\downarrow  ~~~~~~~~~~~~~~~~~~~~~~~~~~~\downarrow \approx ~~~~~  \\
0 \to Ker \varphi~~~~~ \to \bar{H}_n(X;G)~~~~~~~~ ~~\to ~~\varprojlim H_n (X_\lambda;G)~~ ~~~~~~\to ~~~~~ Coker \varphi ~~\to ~~0
   
     \\
\end{matrix}.
\end{equation}
As it is known \cite{4} by Preposition 1.2, for each direct system $ {\bf {A}}=\{ A_\alpha, \pi_{\alpha, \alpha'}, A \}$ of abalian groups there exists a natural exact sequence:
$$0 \to {\varprojlim}^1 Hom(A_\alpha;G)  \to Ext( \varinjlim A_\alpha;G) ~~\to ~~\varprojlim Ext(A_\alpha;G)~~ \to$$
\begin{equation}
 ~~ \to {\varprojlim}^2 Hom(A_\alpha;G) ~~\to ~~0.
\end{equation}
Therefore, for the direct system $ {{H}^*({\bf X})}=\{ {H}^n(X_\lambda), \pi_{\lambda, \lambda'}, \Lambda \}$ of cohomological groups we have:
$$0 \to {\varprojlim}^1 Hom({H}^{n+1}(X_\lambda);G)  \to Ext( \varinjlim H^n(X_\lambda);G) ~~\to ~~\varprojlim Ext(H^n(X_\lambda);G)~~ \to$$
\begin{equation}
 ~~ {\varprojlim}^2 Hom(H^{n+1}(X_\lambda);G) ~~\to ~~0
\end{equation}
and consequently by (2.20) we obtain the exact sequence:
$$0 \to {\varprojlim}^1 Hom(H^{n+1}(X_\lambda);G)  \to \bar{H}_n(X;G) ~~\to ~~\varprojlim H_n(X_\lambda;G)~~ \to$$
\begin{equation}
 ~~ {\varprojlim}^2 Hom(H^{n+1}(X_\lambda);G) ~~\to ~~0.
\end{equation}
On the other hand for each $ \lambda \in \Lambda $ the cohomology group $H^{n+1}(X_\lambda;G)$ is finitely generated \cite{7} and by Corollary 1.5 (see 2e, 3e) we have:
\begin{equation}
{\varprojlim}^2 Hom(H^{n+1}(X_\lambda);G)=0,
\end{equation}
\begin{equation}
{\varprojlim}^1 Hom(H^{n+1}(X_\lambda);G)={\varprojlim}^1 H_{n+1}(X_\lambda; G).
\end{equation}
Therefore, by (2.23), (2.24) and (2.25) we obtain:
\begin{equation}
0 \to {\varprojlim}^1 H_{n+1}(X_\lambda; G)  \to \bar{H}_n(X;G) ~~\to ~~\varprojlim H_n(X_\lambda;G)~~ \to ~~0.
\end{equation}
\end{proof}

\begin{thm} If a homology  theory $\bar{H}_*$  in the Mdzinarishvili sense, defined on the category $Top_C^2$ of compact Hausdorff pairs, is an exact functor of the second argument,  then it is a homology theory in the Inasaridze sense.
\end{thm} 
\begin{proof} We should prove that $\bar{H}_*$ is continuous for each injective group $G$. For each compact Hausdorff space $X \in Top_C^2$ consider an inverse system $ {\bf {X}} = \{ X_{\lambda}, p_{\lambda, \lambda'}, \Lambda \}$ of compact polyhedra such that 
$X={\varprojlim}{\bf {X}}$. By the condition of the theorem, for each abelian group $G$ we have the following short exact sequence:
\begin{equation}
0 \to {\varprojlim}^1 H_{n+1}(X_\lambda; G)  \to \bar{H}_n(X;G) ~~\to ~~\varprojlim H_n(X_\lambda;G)~~ \to ~~0.
\end{equation}
Therefore, we should show that for each  injective group $G$ the first derivative is trivial:
\begin{equation}
{\varprojlim}^1 H_{n+1}(X_\lambda; G)=0.
\end{equation}
Indeed, for each compact polyhedra $X_\lambda$ we have the sequnce \cite{7}:
\begin{equation}
0 \to Ext(H^{n+1}(X_{\lambda});G) \to H_n (X_{\lambda};G) \to Hom(H^n(X_{\lambda});G) \to 0,
\end{equation}
which induces the long exact sequence:

 $0 \to \varprojlim Ext(H^{n+1}(X_{\lambda});G) \to \varprojlim H_n (X_{\lambda};G) \to \varprojlim Hom(H^n(X_{\lambda});G) \to$
 
 $ {\varprojlim}^1 Ext(H^{n+1}(X_{\lambda});G) \to {\varprojlim}^1 H_n (X_{\lambda};G) \to {\varprojlim}^1 Hom(H^n(X_{\lambda});G) \to$ 
 \begin{equation}
{\varprojlim}^2 Ext(H^{n+1}(X_{\lambda});G) \to {\varprojlim}^2 H_n (X_{\lambda};G) \to {\varprojlim}^2 Hom(H^n(X_{\lambda});G) \to  \dots.
 \end{equation}
 As it is known \cite{7}, for each $\lambda \in \Lambda$ the cohomology group $H^n(X_{\lambda};G)$ is finitely generated and by Corollary 1.5 in \cite{4}  we have:
\begin{equation}
{\varprojlim}^r Ext(H^{n+1}(X_{\lambda});G)=0, ~~~r \geq 1.
\end{equation}
Therefore, by (2.30) and (2.31) we obtain the isomorphism:
\begin{equation}
 {\varprojlim}^1 H_n (X_{\lambda};G) \approx {\varprojlim}^1 Hom(H^n(X_{\lambda});G).
 \end{equation}
On the other hand, for each direct system $ {H^*({\bf X})}=\{ H^n(X_\lambda), \pi_{\lambda, \lambda'}, \Lambda \}$ of abelian groups we have:
$$0 \to {\varprojlim}^1 Hom(H^{n+1}(X_\lambda);G)  \to Ext( \varinjlim H^n(X_\lambda);G) ~~\to ~~\varprojlim Ext(H^n(X_\lambda);G)~~ \to$$
\begin{equation}
 ~~ {\varprojlim}^2 Hom(H^{n+1}(X_\lambda);G) ~~\to ~~0.
\end{equation}
By (2.33) for each injective abelian group $G$ we obtain:
\begin{equation}
 ~~ {\varprojlim}^1 Hom(H^{n+1}(X_\lambda);G)=0.
\end{equation}
Therefore, by (2.32) and (2.33) we obtain (2.28). Therefore, for each injective group $G$ we have
\begin{equation}
\bar{H}_n(X;G) \approx \varprojlim H_n(X_\lambda);G)
\end{equation} 

\end{proof}

\subsection*{Acknowledgements}
The authors were supported by grant FR/233/5-103/14 from Shota Rustaveli National Science Foundation (SRNSF).

\end{document}